\documentclass[a4]{amsart}

\usepackage{amssymb}
\usepackage[all]{xy}
\usepackage{amsmath, amsfonts, amsthm , amssymb, amscd}
\input xypic
\usepackage{times}
\newtheorem{theorem}{Theorem}[section]
\newtheorem{proposition}[theorem]{Proposition}
\newtheorem{lemma}[theorem]{Lemma}
\newtheorem{corollary}[theorem]{Corollary}

\theoremstyle{definition}

\theoremstyle{remark}
\newtheorem{remark}[theorem]{Remark}

\newcommand{\Hom}{\ensuremath{\mathrm{Hom}}}
\newcommand{\Z}{\ensuremath{\mathbb{Z}}}
\newcommand{\Q}{\ensuremath{\mathbb{Q}}}
\renewcommand{\L}{\ensuremath{\mathcal{L}}}
\DeclareMathOperator{\rk}{rank}
\DeclareMathOperator{\dummygg}{\mathfrak{g}}
\renewcommand{\gg}{\dummygg}
\DeclareMathOperator{\hh}{\mathfrak{h}}
\DeclareMathOperator{\TT}{\mathfrak{t}}
\DeclareMathOperator{\s}{\mathfrak{s}}
\newcommand{\lra}{\longrightarrow}
\newcommand{\dr}[3]{\ensuremath{#1\stackrel{#2}
{\longrightarrow}#3}}
\newcommand{\ddr}[5]{\ensuremath{#1\stackrel{#2}
{\longrightarrow}#3\stackrel{#4}{\longrightarrow}#5}}

\begin{document}
\title{The integral homology ring of the based loop space on some generalised symmetric spaces}
\author{Jelena Grbi\'{c} and Svjetlana Terzi\'c}
\address{School of Mathematics, University of Manchester,
         Manchester M13 9PL, United Kingdom}
\address{Faculty of Science, University of Montenegro,D\v zord\v za Va\v singtona bb,
         81000 Podgorica, Montenegro}
\email{Jelena.Grbic@manchester.ac.uk}
\email{sterzic@ac.me}
\date{}

\begin{abstract}
In this paper we calculate the integral Pontrjagin homology ring of the based loop space on some generalised symmetric spaces with a toral stationary subgroup. In the Appendix we show that the method can be applied to other type generalised symmetric spaces as well.

\end{abstract}

\maketitle


\section{Introduction}
A pointed topological space $X$ with multiplication $\mu\colon X\times X\lra X$ is called an $H$-space. The multiplication induces a ring structure in homology $H_*(X)$. A based loop space $\Omega X$ is an $H$-space as one of the possible multiplications is given by loop concatenation. The ring structure in $H_*(\Omega X)$ induced by loop multiplication is called the Pontrjagin homology ring.  In this paper we compute the integral Pontrjagin homology ring of the based loop space on some generalised symmetric spaces $G/S$, where $G$ is a simple compact Lie group $G$ having the torus $S$ as a stationary subgroup. More precisely, we consider generalised symmetric spaces which have zero Euler characteristic, that is, for which the torus $S$  is not maximal in $G$.  Thus this paper can be considered as a natural sequel to  paper~\cite{GT}, where the authors computed the integral Pontrjagin homology of the based loop space of complete flag manifolds of compact simple Lie groups.

Generalised symmetric spaces $G/H$ are defined by the condition that their stationary subgroup $H$ is the fixed point subgroup of a finite order automorphism of the group $G$. These spaces consequently admit finite order symmetries making them to have rich geometry and thus attract constant attention of many geometry focused research starting with~\cite{WG1},\cite{WG2} until recent once, see for example~\cite{AA1},\cite{B},\cite{KT}. Generalised symmetric spaces play important role in the theory of homogeneous spaces and, thus, in many areas of mathematics and  physics such as representation theory, combinatorics, string topology. In particular, complex homogeneous spaces $G/H$, which are important examples in complex cobordisms and theory of characteristic classes, were characterised by Passiencier~\cite{P} as the homogeneous spaces of positive Euler characteristic whose isotropy subgroup $H$ is a fixed point subgroup for some odd order finite subgroup of the inner automorphisms of the group $G$ and thus are complex generalised symmetric spaces.

The methods used in this paper are analogous to that of~\cite{GT}. By~\cite{KT} and~\cite{TFP} all generalised symmetric spaces are Cartan pair homogeneous spaces and, thus, formal in the sense of Sullivan (see Subsections~\ref{Coh},~\ref{MinMod}). Therefore, starting from their rational cohomology algebras  and using the method of rational homotopy theory together with the Milnor and Moore theorem which expresses the rational Pontrjagin homology algebra in terms of the universal enveloping algebra of the rational homotopy Lie algebra, we compute the rational homology algebras for the based loop space on generalised symmetric spaces $G/S$.  We prove that  the integral loop homology groups of these spaces have no torsion which together with the rational computations enables us to determine their integral Pontrjagin ring structure as well.  In particular, we calculate the integral Pontrjagin homology ring of the based loop spaces on the following generalised symmetric spaces: $SU(2n+1)/T^n$,  $SU(2n)/T^n$,  $SO(2n+2)/T^n$,  $SO(8)/T^2$, and $E_6/T^4$.

To illustrate that our approach is more general, in the Appendix we include the analogous calculations for the generalised symmetric space $U(n)/(T^k\times U(n-k))$ which obviously is not obtained by taking a simple compact Lie group and quotienting out a toral subgroup.

\section{Torsion in loop space homology}
\begin{proposition}\label{simplyconnected}
Let $G$ be a compact connected Lie group and $H$ its closed connected subgroup such that $G/H$ is simply connected. Then $H_{*}(\Omega (G/H);\Z )$ is torsion free if and only if $H_{*}(\Omega (G/(T^k\times H));\Z )$ is torsion free, whenever $T^k\times H$ is a subgroup of $G$, where $k\leq \rk G-\rk H$.
\end{proposition}
\begin{proof}
There is a fibration sequence
\[
\Omega (G/H)\lra \Omega (G/(T^k\times H))\lra \dr{T^k}{*}{G/H}\lra G/(T^k\times H)
\]
which implies the homotopy decomposition
\[
\Omega (G/(T^k\times H))\simeq T^k\times \Omega( G/H)
\]
and the proof of the proposition. It is important to remark that although the spaces  $\Omega (G/(T^k\times H))$ and $T^k\times \Omega( G/H)$ are $H$-spaces, the decomposition is given only in terms of topological spaces ignoring their $H$-space structure, that is, this decomposition is not realised by an $H$-equivalence (one that preserves the multiplicative structures of the spaces involved).
\end{proof}
\begin{corollary}
Let $G$ be a compact connected simple Lie group and $S$ a toral subgroup in $G$. Then $H_*(\Omega (G/S);\mathbb Z)$ is torsion free.
\end{corollary}
\begin{proof}
In~\cite{GT} the authors showed that if $T$ is a maximal torus in a compact connected Lie group then $H_*(\Omega(G/T);\Z)$ is torsion free. For $G$ simply connected, $G/S$ is simply connected as well and the statement follows by Proposition~\ref{simplyconnected}. 

The only compact connected simple Lie group which is not simply connected is $SO(n)$. We show that $SO(n)/T^k$, where $1\leq k\leq [\frac{n}{2}]$ is simply connected by induction. There is a fibration sequence  
\[
SO(2k)/T^{k}\lra SO(n)/T^k\lra SO(n)/SO(2k) .
\]
Since  $\pi _{i}(V_{n, 2k})=0$ for $0\leq i\leq 2k$ and $SO(2k)/T^k$ is simply connected (see for example~\cite{MT},~\cite{GT}), we conclude that $SO(n)/T^k$ is simply connected. Thus the statement for $\Omega (SO(n)/S)$ holds.
\end{proof}

\section{The rational homology of loop spaces}\label{s-rational}

In this section we compute rational homology of the based loop space on  generalised symmetric space $G/S$ of a simple compact Lie group $G$ having toral isotropy subgroup $S$ and zero Euler characteristic. In these generalised symmetric spaces, $S$ is a non-maximal toral subgroup, that is, $\rk S<\rk G$, which can be obtained as the stationary subgroup of an outer automorphism of the group $G$.  These generalised symmetric spaces are listed in~\cite{TFP} as: $SU(2n+1)/T^n$,  $SU(2n)/T^n$,  $SO(2n+2)/T^n$,  $SO(8)/T^2$, and $E_6/T^4$. They are obtained by   the outer automorphisms of the groups $SU(2n+1)$, $SU(2n)$, $SO(2n+2)$, $SO(8)$ and $E_6$ having respectively the orders $2n+1, 2n-1, 2n, 12$ and $18$. 

\subsection{The rational cohomology of homogeneous spaces.}\label{Coh} We recall some classical results on the cohomology of homogeneous spaces with rational coefficients. Consider a homogeneous space $G/H$, where $G$ is a compact connected Lie group and $H$ its connected closed subgroup.
The Hopf theorem~\cite{Borel} states that the 
rational cohomology algebra of $G$ is an exterior algebra
\[
H^{*}(G,\Q )\cong \wedge (z_1,\ldots ,z_n)
\]
where $z_1,\ldots, z_n$ are the universal transgressive generators in degree $\deg z_i=2k_i-1$, where $k_1,\ldots, k_n$ are the exponents of the group $G$ (see~\cite{cox}) and $n = \rk G$ denotes the dimension of the maximal torus in $G$. 
Recall~\cite{Borel} that for a fibration $F\lra E\lra B$, we can define a map $\tau$ called transgression which maps a subgroup of $H^*(F)$ to a quotient of $H^*(B)$. If we consider the universal $G$-bundle $G\lra EG\lra BG$, then elements $x\in H^*(G)$ on which the transgression is defined are said to be universal transgressive.

The rational cohomology algebra of the classifying space $BG$ of a group $G$ is the algebra of polynomials on the maximal abelian subalgebra $\TT$ of the Lie algebra $\gg$ for $G$ which are 
invariant under the action of the Weyl group $W_{G}$, that is,
$$H^{*}(BG,\Q )\cong \Q [\TT ]^{W_{G}}.$$ 
This algebra is generated by the polynomials $P_1,\ldots ,P_n$ called Weyl invariant generators, which correspond to $z_1,\ldots ,z_n$ by transgression in the universal bundle for $G$ and thus have degree $\deg P_i = 2k_i$, $1\leq i\leq n$ (see~\cite{Borel}).

Denote by $\s \subset \TT$ the maximal abelian subalgebra of the Lie algebra $\hh$ for $H$. Since $W_{H}\subset W_{G}$, the polynomials from $H^{*}(BG,\Q )$ when restricted to $\s$ belong to $H^{*}(BH,\Q )$.

The {\it Cartan algebra} for a homogeneous space  $G/H$ (see~\cite{Borel}) is a differential graded algebra $(C,d)$ defined by:
$$
C=H^{*}(BH,\Q)\otimes H^{*}(G,\Q),\;\;  
d(b\otimes 1)=0, \;\; d(1\otimes z_{i})= \rho ^{*}(P_{i})\otimes 1
$$
where $\rho ^{*} : \Q [\TT ]^{W_{G}}\to \Q [\s ]^{W_{H}}$ denotes the restriction. 

By  the famous Cartan theorem~\cite{Borel}, the Cartan algebra determines the rational cohomology of $G/H$, that is,
$$H^{*}(G/H,\Q )\cong H^{*}(C, d).$$

There is a  wide class of homogeneous spaces, called Cartan pair homogeneous spaces in terminology of~\cite{GHV} or normal position homogeneous spaces in terminology of~\cite{DK}, which behave nicely from the point of view of both rational cohomology and  rational homotopy  theory.
A homogeneous space $G/H$ is a Cartan pair homogeneous space if one can choose $n$ algebraically independent generators $P_1,\ldots ,P_n \in \Q [\TT]^{W_G}$ such that $\rho ^{*}(P_{r+1}),\ldots ,\rho ^{*}(P_n)$ belong to the ideal in $\Q [\s ]^{W_H}$  generated by $\rho ^{*}(P_1),\ldots ,\rho^{*}(P_r)$. Furthermore, when this is the case  one can choose $P_{r+1},\ldots ,P_{n}$ such that $\rho ^{*}(P_{r+1})=\ldots =\rho ^{*}(P_n)=0$. Then the Cartan theorem directly implies that the rational cohomology algebra for these spaces is given by
\begin{equation}\label{cpc}
H^{*}(G/H;\Q) \cong H^{*}(BH,\Q)/\left\langle \rho ^{*}(P_1),\ldots ,\rho^{*}(P_r) \right\rangle \otimes \wedge (z_{r+1},\ldots z_n).
\end{equation}
In this case the sequence $\rho ^{*}(P_1),\ldots ,\rho ^{*}(P_r)$ is regular, that is, the class $[\rho ^{*}(P_i)]$ in\\ $H^*(BH;\Q)/(\rho ^{*}(P_1),\dots, \rho ^{*}(P_{i-1})$ is not a zero divisor. We can further assume that algebra~\eqref{cpc} is reduced. Thus we can consider only those elements of the regular sequence which are decomposable, eliminating the generators in $H^{*}(BH,\Q )$  which are linear combinations of $\rho ^{*}(P_1),\ldots ,\rho ^{*}(P_r)$.

Among examples of Cartan pair homogeneous spaces are homogeneous spaces of positive Euler characteristic,
compact symmetric spaces~\cite{Borel},~\cite{Takeuchi} and generalised symmetric spaces~\cite{TFP}.

\subsection{On minimal model theory}(see~\cite{FHT}) \label{MinMod} 
Let $(A, d_{A})$ be a commutative graded differential algebra over the real numbers.  A differential graded algebra 
$(\mu _{A}, d)$ is called {\it minimal model} for $(A, d_{A})$ if
\begin{itemize}
\item [(i)] there exists differential graded algebra morphism $h_{A}\colon(\mu _{A}, d)\lra (A, d_{A})$ inducing an isomorphism in their cohomology algebras (such $h_{A}$ is called quasi-isomorphism);
\item[(ii)] $(\mu _{A}, d)$ is a free algebra in the sense that $\mu _{A} = \wedge V$ is an exterior algebra over graded vector space $V$;
\item[(iii)] differential $d$ is indecomposable meaning that for a fixed set $V =\{ P_{\alpha},\alpha \in I \}$ of free generators of $\mu _{A}$  for any $P_{\alpha}\in V$, $d(P_{\alpha})$ is a polynomial in generators $P_{\beta}$ with 
no linear part.
\end{itemize} 

Two algebras are  said to be {\it weakly equivalent} if there exists quasi-isomorphism between them. This is equivalent to say that these algebras have isomorphic minimal models. The algebra $(A, d_{A})$ is said to be {\it formal} if it is weakly equivalent to the algebra $(H^{*}(A), 0)$.

For a smooth connected  manifold $M$, the minimal model is by definition the minimal model of its de Rham algebra  of differential forms $\Omega _{DR}(M)$. In the case when $M$ is simply connected manifold its minimal model completely classifies its rational homotopy type. The manifold $M$ is said to be formal (in the sense of Sullivan) if $\Omega _{DR}(M)$ is a formal algebra.

It is a classical result~\cite{ON} that compact homogeneous space  is formal if and only if it is a Cartan pair homogeneous space. In particular, all generalised symmetric spaces are formal in the sense of Sullivan~\cite{KT}. Thus the minimal model for Cartan pair homogeneous space $G/H$(see~\cite{BG}) is given by $\mu =(\wedge V, d)$ where
\begin{equation}
V= (u_1,\ldots ,u_l, v_1,\ldots ,v_k,z_{r+1},\ldots ,z_{n})
\end{equation}
\begin{equation} 
d(u_i)=0,\;\; d(v_j)=\rho ^{*}(P_j),\;\; d(z_i)=0
\end{equation}
with $u_i$ correspond to the remaining generators of $H^{*}(BH;\Q )$, while $v_j$ correspond to the remaining elements of the sequence $\rho ^{*}(P_1),\ldots ,\rho ^{*}(P_r)$.

\subsection{The loop space rational homology.}
Let $\mu=(\Lambda V, d)$ be a  Sullivan minimal model of a simply connected topological space $M$ with the rational homology of finite type.  Then $d\colon V\lra \Lambda^{\geq 2}V$ can be
decomposed as $d=d_1+d_2+\cdots$, where $d_i\colon V\lra\Lambda^{\geq i+1}V$. In particular, $d_1$ is called the \emph{quadratic part} of the differential $d$.

The homotopy Lie algebra $\L$ of $\mu$ is defined in the following
way. The underlying  graded vector space $L$ is given by 
\[
sL=\Hom(V, \Q)
\]
where  $sL$ denotes the usual suspension  defined by $(sL)_i=(L)_{i-1}$.
We can define a pairing $\langle \, ; \, \rangle\colon V\times sL
\lra \Q$ by $\langle v;sx\rangle =(-1)^{\deg v}sx(v)$ and extend it
to $(k+1)$-linear maps
\[
\Lambda^kV\times sL\times\cdots\times sL\lra \Q
\]
by letting
\[
\langle v_1\wedge\cdots\wedge v_k; sx_k,\ldots
,sx_1\rangle=\sum_{\sigma\in S_k}\epsilon_\sigma\langle
v_{\sigma(1)}; sx_1\rangle\cdots\langle v_{\sigma(k)};sx_k\rangle
\]
where $S_k$ is the symmetric group  and $\epsilon_{\sigma}=\pm 1$  are determined by
\[
v_{\sigma(1)}\wedge\cdots\wedge v_{\sigma(k)}=\epsilon_\sigma
v_1\wedge\cdots\wedge v_k.
\]
Then the space $L$
inherits a Lie bracket $[\, ,
 \, ]\colon L\times L\lra L$ from $d_1$ uniquely determined by
 \begin{equation}
\label{bracket}
 \langle v; s[x,y]\rangle= (-1)^{\deg y+1}\langle
 d_1v;sx,sy\rangle \quad \text{ for } x,y\in L, v\in V.
 \end{equation}
Denote by $\L$ the Lie algebra $(L, [\, ,\, ])$.

On the other hand in the category of topological spaces and continuous maps, we can define the Samelson products $[f,g]\colon S^{p+q}\lra\Omega M$ of maps $f\colon S^p\lra\Omega M$ and $g\colon S^q\lra \Omega M$ by the composite
\[
\ddr{S^p\wedge S^q}{f\wedge g}{\Omega M\wedge\Omega M}{c}{\Omega M}
\]
where $c$ is given by the multiplicative commutator, that is, $c(x,y)=x\cdot y\cdot x^{-1}\cdot y^{-1}$. Recall that  there is a  graded Lie algebra $L_M=(\pi_*(\Omega M)\otimes \Q;
[\, , \,])$ called the {\it rational homotopy Lie algebra of
$M$}, for which  the commutator $[\, , \,]$ is given by the Samelson product.
There is an isomorphism between the rational homotopy Lie algebra
$L_M$ and the homotopy Lie algebra $\L$ of $\mu$. Milnor and Moore [see Appendix in ~\cite{MM}] showed that for a path connected homotopy associative $H$-space with unit $G$, there is an isomorphism of Hopf algebras $U(\pi_*(G)\otimes\Q)\cong H_*(G; \Q)$. As loop multiplication is homotopy associative with unit, applying the Milnor and Moore theorem to our case,  it follows
that
\[
H_*(\Omega M; \Q)\cong U\L
\]
where $U\L$ is the universal enveloping algebra for $\L$. Further
on,
\[
U\L \cong T(L)/\langle xy-(-1)^{\deg x\deg y}yx-[x,y]\rangle.
\]
For a more detailed account of this construction see for
example~\cite{FHT}, Chapters  $12$ and $16$.

\subsection{The loop space of some generalised symmetric spaces.}

Recall from~\cite{TFP} that the rational cohomology of $SU(2n+1)/T^n$ is given by
\[
H^*(SU(2n+1)/T^n;\Q)\cong \Q [x_1,\ldots ,x_{n}]/\langle
P_2,P_4,\ldots ,P_{2n} \rangle \otimes \wedge (z_3, z_5,\ldots ,z_{2n+1}) ,
\]
where $P_{2i}=\sum_{j=1}^{n}x_j^{2i}$, $1\leq i\leq n$ and $\deg x_j=2$, $\deg z_{2j+1}=4j+1$.
Since  $SU(2n+1)/T^n$ is formal, the minimal model for $SU(2n+1)/T^n$ is the  minimal model for the commutative differential graded algebra $(H^*(M;\Q), d=0)$. It is given by
$\mu=(\Lambda V, d)$, where
\[
V=(x_1,\ldots ,x_{n},y_1,\ldots ,y_n,z_1,\ldots ,z_n) \] with $\deg u_{i}=2$,
$\deg y_i=4i-1$ and $\deg z_i=4i+1$ for  $1\leq i\leq n \ $.

The differential $d$ is defined by
\begin{equation}\label{differential}
d(x_{i})=d(z_{i})=0, \quad d(y_{i})=\sum _{j=1}^{n}x_{j}^{2i}\ .
\end{equation}
\begin{theorem}
The rational homology ring of the loop space on the 
manifold $SU(2n+1)/T^n$ is
\begin{equation}\label{QSU(2n+1)} 
H_*(\Omega ( SU(2n+1)/T^n); \Q) \cong
\end{equation}
\[
\indent \left( T(a_1,\ldots ,a_n)/\left\langle
a_{1}^2=\ldots =a_n^2, a_ia_j=-a_ja_i\, |\, 1\leq i,j\leq n \right\rangle\right)
\otimes \Q[b_2,\ldots ,b_n,c_1,\ldots ,c_n]
\]
where the generators $a_{i}$ are of degree $1$ for $1\leq i\leq n$,
the generators $b_{j}$ are of degree $4j-2$ for $2\leq j\leq n$, and the generators
$c_k$ are of degree $4k$ for $1\leq k\leq n$.
\end{theorem}

\begin{proof}
The underlying vector space of the homotopy Lie algebra $\L$ of
$\mu$ is given by
\[
L=(a_1,\ldots ,a_n,b_1,\ldots ,b_n,c_1,\ldots ,c_n)\] where $\deg(a_i)=1,\
\deg(b_i)=4i-2$ and $\deg (c_i)=4i$ for $1\leq k\leq n$ .

In order to define Lie brackets we need the quadratic part $d_1$ of
the differential in the minimal model. In this case, using the
differential $d$ defined in~\eqref{differential}, the quadratic part
$d_1$ is given by
\[
d_1(x_i)=d_1(z_i)=0  \quad  d_1(y_1)=2\sum _{j=1}^{n}x_j^2  d_1(y_j)=0 \ \text{ for } k\neq 1 \ .
\]

By the defining property of the Lie bracket stated in~\eqref{bracket}, we have

\[
\langle y_1, s[a_i,a_i]\rangle=\left\langle \sum x_j^2; sa_i, sa_i\right\rangle =\langle x_i^2; sa_i,sa_i\rangle =2
\]
\[
[a_i, b_j]=[a_i, c_j]=[b_i, b_j]=[b_i, c_j]=[c_i,c_j]=0 \text{ for } 1\leq i,j\leq n \text{ and}
\]
\[
[a_i,a_j]=0 \text{ for } i\neq j\ .
\]
Therefore in the tensor algebra $T(a_1,\ldots , a_n,b_1,\ldots
,b_n,c_1,\ldots ,c_n)$, the Lie brackets above induce the following relations
\[\begin{array}{ll}
a_ia_j+a_ja_i =0 & \text{ for }  1\leq i,j\leq n, i\neq j \\
a_i^2=b_1 & \text{ for }  1\leq i\leq n\\
a_ib_j=b_ja_i &  \text{ for }  1\leq i,j\leq n \\
a_ic_j=c_ja_i &  \text{ for }  1\leq i,j\leq n\\
b_ib_j=b_jb_i &  \text{ for }  1\leq i,j\leq n\\
b_ic_j=c_jb_i &  \text{ for }  1\leq i,j\leq n\\
c_ic_j=c_jc_i &  \text{ for }  1\leq i,j\leq n.  
\end{array}
\]
This proves the theorem.
\end{proof}
Since the rational cohomology algebra for $SU(2n)/T^n$ is given by
\[
H^*(SU(2n)/T^n;\Q)\cong \Q [x_1,\ldots ,x_{n}]/\langle
P_2,P_4,\ldots ,P_{2n} \rangle \otimes \wedge (z_3, z_5,\ldots ,z_{2n-1}) ,
\]
where $P_{2i}=\sum_{j=1}^{n}x_j^{2i}$, $1\leq i\leq n$ and $\deg x_j=2$, $\deg z_{2j+1}=4j+1$ for $1\leq j\leq n-1$
in the same way as in Theorem 3.1 we prove the following.
\begin{theorem}
The rational homology of the loop space on the 
manifold $SU(2n)/T^n$ is
\begin{equation}\label{QSU(2n)} 
H_*(\Omega ( SU(2n)/T^n); \Q) \cong
\end{equation}
\[
\indent \left( T(a_1,\ldots ,a_n)/\left\langle
a_{1}^2=\ldots =a_n^2, a_ia_j=-a_ja_i\, |\, 1\leq i,j\leq n \right\rangle\right)
\otimes \Q[b_2,\ldots ,b_n,c_1,\ldots ,c_{n-1}]
\]
where the generators $a_{i}$ are of degree $1$ for $1\leq i\leq n$,
the generators $b_{j}$ are of degree $4j-2$ for $2\leq j\leq n$, and the generators
$c_k$ are of degree $4k$ for $1\leq k\leq n-1$.\qed
\end{theorem}
\begin{theorem}\label{SO(2n+2)}
\label{rationalgeneralsymm} The rational homology of the based loop
space on $SO(2n+2)/T^n$ is given by
\[
H_*(\Omega (SO(2n+2)/T^n); \Q)\cong
\]
\[
\left(T(a_1,\ldots ,a_n)/\left\langle \begin{array}{l}a_1^2=\ldots =a_n^2,\\
a_ka_l=-a_la_k \text{ for } k\neq
l\end{array}\right\rangle\right)\otimes \Q[b_2,\ldots , b_n,
b_{n+1}]
\]
where the generators $a_{i}$ are of degree $1$ for $1\leq i\leq n$,
the generators $b_{k}$ are of degree $4k-2$ for $2\leq k\leq n$,
and the generator $b_{n+1}$ is of degree $2n$.
\end{theorem}
\begin{proof}
The rational cohomology of $SO(2n+2)/T^n$ is given by
\[
 H^*(SO(2n+2)/T^n; \Q)\cong \big(\Q[x_1,\ldots ,x_n]/\langle P_2, P_4,\ldots ,
P_{2n}\rangle\big)\otimes \Lambda (z)
\]
where $P_{2k}=\sum^n_{i=1}x_i^{2k}$, $\deg (x_i)=2$ for $1\leq i\leq n$, $\deg (P_{2k})=4k$ for $2\leq k\leq n$, and $\deg (z)=2n+1$.
The minimal model for $SO(2n+2)/T^n$ is the minimal model for the
commutative differential algebra $(H^*(SO(2n+2)/T^n; \Q), d=0)$. It
is given by $\mu=(\Lambda V, d)$, where
\[
V=(x_1,\ldots ,x_{n},y_1,\ldots ,y_n, y_{n+1}),
\]
$\deg (x_{k})=2, \deg (y_k)=4k-1$ for $1\leq k\leq n$ and $\deg
(y_{n+1})=2n+1$.

The differential $d$ is given by
\begin{equation}\label{differential2}
d(x_{k})=0, \ d(y_{k})=P_{2k}=\sum _{i=1}^{n}x_{i}^{2k} \text{ for }
1\leq k\leq n \text{ and } d(y_{n+1})=0\ .
\end{equation}
Now the underlying vector space of the homotopy Lie algebra $L$ of
$\mu$ is given by
\[
L=(a_1,\ldots ,a_n,b_1,\ldots ,b_n, b_{n+1})
\]
where $\deg (a_k)=1,\ \deg (b_k)=4k-2$ for $1\leq k\leq n$ and $\deg
(b_{n+1})=2n$. In order to define Lie brackets we need the quadratic part $d_1$ of
the differential $d$ defined in~\eqref{differential2}. It is given by
\[
d_1(x_k)=0, \ d_1(y_1)=\sum _{i=1}^{n} x_i^2,\ d_{1}(y_k)=0 \ \text{
for }2\leq k\leq n \text{ and } d_1(y_{n+1})=0 .
\]
For dimensional reasons, we have
\[
[a_k, b_l]=[b_s, b_l]=0 \text{ for } 1\leq k\leq n \text{ and }
1\leq s,l\leq n+1.
\]
By the defining property of the Lie bracket stated
in~\eqref{bracket}, we have
\[
\langle y_1, s[a_k,a_k]\rangle=\left\langle \sum x_i^2; sa_k,
sa_k\right\rangle =\langle x_k^2; sa_k,sa_k\rangle =1 \text{ and }
\]
\[
\langle y_1, s[a_k,a_l]\rangle=0 \text{ for } k\neq l \
\]
resulting in the non-trivial commutators
\[
[a_k,a_k]=2b_1 \text{ for } 1\leq k\leq n \ .
\]
Therefore in the tensor algebra $T(a_1,\ldots , a_n,b_1,\ldots ,b_n,
b_{n+1})$ the Lie brackets above induce the following relations
\[
\begin{array}{ll}
  b_kb_l=b_lb_k&  \text{ for }  1\leq k,l\leq n+1 \\
  2a^2_k = b_1 & \text{ for } 1\leq k\leq n\\
  a_ka_l=-a_la_k & \text{ for } 1\leq k,l\leq n\\
  a_kb_l=b_la_k & \text{ for } 1\leq k\leq n \text{ and }1\leq l\leq
  n+1.
\end{array}
\]
Thus
\begin{equation}
 U\L  = \left( T(a_1,\ldots ,a_n)/\left\langle \begin{array}{l}a_1^2=\ldots =a_n^2,\\
a_ka_l=-a_la_k \text{ for } k\neq l\end{array}\right\rangle
\right)\otimes \Q[b_2,\ldots ,b_n, b_{n+1}] \ .
\end{equation}
This proves the theorem.
\end{proof}
 \begin{theorem}
The rational homology ring of the loop space on the 
manifold $SO(8)/T^2$ is
\begin{equation}\label{QSO(8)} 
H_*(\Omega ( SO(8)/T^2); \Q) \cong
\end{equation}
\[
\indent \left( T(a_1, a_2)/\left\langle
a_{1}^2=a_2^2= a_1a_2+a_2a_1 \right\rangle\right)
\otimes \Q[b_2,c_1,c_2]
\]
where the generators $a_{i}$ are of degree $1$ for $i=1,2$,
the generators $b_{2}$ is of degree $10$,  and the generators
$c_1,c_2$ are of degree $6$.
\end{theorem}

\begin{proof}
The rational cohomology for $SO(8)/T^2$  (see for example~\cite{TFP}) is given by
\[
H^{*}(SO(8)/T^2;\Q )\cong \Q[x_1,x_2]/\langle x_1^2+x_2^2+x_1x_2, (x_1+x_2)^2x_1^2x_2^2\rangle \otimes \wedge (z_2,z_4)
\]
implying that its minimal model is of the form $(\wedge V, d)$ where
$V=(x_1,x_2,y_1,y_2,z_1,z_2)$, $\deg x_i=2$, $\deg y_1=3$, $\deg y_2=11$, $\deg z_1=\deg z_2=7$ and
\[ d(x_i)=d(z_i)=0, \ d(y_1)=x_1^2+x_2^2+x_1x_2, \ d(y_2)=(x_1+x_2)^2x_1^2x_2^2 \ .
\]
Therefore its rational homotopy Lie algebra is given by $L=(a_1,a_2,b_1,b_2,c_1,c_2)$, $\deg a_i=1$, $\deg b_1=2$, $\deg b_2=10$, $\deg c_1=\deg c_2=6$, and
\[
[a_i,b_j]=[a_i,c_j]=[b_i,b_j]=[b_i,c_j]=[c_i,c_j]=0, \ [a_1,a_1]=[a_2,a_2]=2b_1,\ [a_1,a_2]=b_1 \ .
\]
 
These commutators imply the needed relations in $U\L$ and thus we obtain the ring structure of the rational homology for the based loop space on $SO(8)/T^2$.
\end{proof}
\begin{theorem}
The rational homology of the loop space on the 
manifold $E_6/T^4$ is given by
\begin{equation}
H^{*}(E_6/T^4;\Q )\cong
\end{equation}
\[
\indent \left (T(a_1,a_2,a_3,a_4)\left\langle a_1^2=a_2^2=a_3^2=a_4^2, a_ia_j=-a_ja_i\text { for } i\neq j \right\rangle \right )\otimes \Q [b_4,b_5,b_7,b_8,b_{11}],
\]
where $\deg a_i=1$, $1\leq i\leq 4$ and $\deg b_j=2j$, $j=4,5,7,8,11$.
\end{theorem}

\begin{proof}
Let $T^6$ be the maximal torus in $E_6$ containing $T^4$. The Weyl invariant generating polynomials for $H^{*}(BE_6;\Q )$ may be taken to be $P_{k}=\sum_{i=1}^{6}(x_i\pm \epsilon )^k +\sum_{1\leq i<j\leq 6}(-x_i-x_j)^k$ for $k=2,5,6,8,9,12$. Here $x_1,\ldots ,x_6$ and $\epsilon$ denote the  canonical coordinates for $T^6$ in $E_6$ where $\sum_{i=1}^{6}x_i=0$ (see for example~\cite{ON}).
The rational cohomology algebra for $E_6/T^4$ is in~\cite{TMZ} given by
\[
H^{*}(E_6/T^4;\Q ) = H^{*}(BT^4;\Q )/\left\langle \rho ^{*}(P_2), \rho ^{*}(P_6),\rho ^{*}(P_8),\rho ^{*}(P_{12})\right\rangle\otimes \wedge (z_5,z_9)
\]
where $\rho ^{*}$ denotes the restriction from the maximal torus $T^6$ in $E_6$ to the torus $T^4$.  If $H_i$, $1\leq i\leq 6$ is Shevalley basis for  $T^6$, then the basis for $T^4$ in the case of generalised symmetric space $E_6/T^4$  is by~\cite{TFP} given by $\bar{H_1}=H_1+H_5$, $\bar{H_2}=H_2+H_4$, $\bar{H_3}=H_3$ and $\bar{H_4}=H_6$. Therefore it follows from~\cite{TMZ} that when restricted to $T^4$, $x_1,\ldots ,x_6$ satisfy  $x_1=-x_6$, $x_2=-x_5$, $x_3=-x_4$. 

Thus the minimal model for $E_6/T^4$ is given by
\[
V=(u_1,u_2,u_3,u_4,v_1,v_4,v_5,v_7,v_8, v_{11})
\]
where $\deg u_i=2$, $deg v_j=2j+1$, and $d(u_i)=0$, $d(v_4)=d(v_8)=0$ while $d(v_i)=\rho ^{*}(P_{2i+2})$ for $i=1,5,7,11$. It further implies that the quadratic part $d_1$ of the differential $d$ vanishes except $d_1(v_1)=12(u_1^2+u_2^2+u_3^2+u_4^2)$. Therefore for  the homotopy Lie algebra of this minimal model we obtain 
\[
L=(a_1,a_2,a_3,a_4,b_1,b_4,b_5,b_7,b_8,b_{11})
\]
where $\deg a_i=1$, $\deg b_j=2j$ and the brackets are given by
\[
[a_i,a_j]=[a_i,b_j]=[b_k,b_l]=0\;\; [a_i,a_i]=24b_1.
\]
It implies that $a_1^2=a_2^2=a_3^2=a_4^2=12b_1$ and $a_ia_j+a_ja_i=0$ for $i\neq j$.
\end{proof}   
 
 \section{Integral Pontrjagin homology}

In this section we study the integral Pontrjagin ring structure of generalised symmetric spaces 
$\Omega (G/S)$, where $G$ is a simple Lie group and $S$ an appropriate toral subgroup whose
rational homology we calculated in the previous
section. In addition to the rational homology calculation we make use of  the results from~\cite{BT}  and~\cite{N} on
integral homology of the identity component $\Omega _{0}G$ of the
loop space on $G$. Recall that for a compact connected Lie group $G$, the Pontrjagin homology  $H_{*}(\Omega _{0}G;\Q)$ is primitively generated, that is, it is generated as an algebra by its space of primitive elements.
It is well known that if $G$ is a simply connected Lie group, then
$\pi_{2}(G/S)\cong \Z ^{\dim S}$ and $\pi _{3}(G/S)\cong \Z$. Let $\alpha,\beta\colon S^2\lra G/S$ and let $\widetilde{\alpha}, \widetilde{\beta}\colon S^1\lra \Omega(G/S)$ denote their adjoints, respectively. The adjoint $W(\alpha, \beta)$ of the Samelson product $[\widetilde{\alpha},\widetilde{\beta}]\colon S^2\lra\Omega(G/S)$ is called the Whitehead product of $\alpha$ and $\beta$. Let 
\[
W\colon \pi _{2}(G/S)\otimes \pi _{2}(G/S)\to \pi _{3}(G/S)
\]
denote also the pairing given by the Whitehead product. In what follows,
we identify $H_{1}(S;\Z )$ with $\pi _{2}(G/S)$ and $H_{2}(\Omega
G;\Z )$ with $\pi _{3}(G/S)$ via natural homomorphisms. Thus since
there is no torsion in homology, and using the rational homology 
results from Section~\ref{s-rational}, we obtain that there is a split extension of
algebras
\[
\xymatrix{ 1 \ar[r] & H_*(\Omega G; \Z)\ar[r] &
H_*(\Omega(G/S);\Z )\ar[r] &H_*(S;\Z )\ar[r] &1}
\]
with the extension given by $[\alpha ,\beta]=W(\alpha ,\beta)\in H_{2}(\Omega G;\Z )$, where
$\alpha ,\beta \in H_{1}(S;\Z )$.

\subsection{The integral homology of $\Omega (SU(2n+1)/T^n)$ and $\Omega (SU(2n)/T^n$}
\begin{theorem}
\label{SU2n1integral} The integral Pontrjagin homology ring of the loop
space on $SU(2n+1)/T^n$ is
\[
H_*\left(\Omega (SU(2n+1)/T^n); \Z\right)\cong
\]
\[
\left( T(x_1,\ldots ,x_n)\otimes \Z[y_2,\ldots
,y_n, z_1,\ldots, z_n]\right)/\left\langle x_1^2=\ldots =x_n^2 , x_px_q+x_qx_p \text{ for }
1\leq p,q\leq n, p\neq q\right\rangle
\]
where the generators $x_1,\ldots ,x_n$ are of degree $1$, the
generators $y_i$ are of degree $4i-2$ for $2\leq i\leq n$,  the
generators $z_i$ are of degree $4i$ for $1\leq i\leq n$.
\end{theorem}
\begin{proof}
We explain the extension of the algebra in more detail for the generalised symmetric space $SU(2n+1)/T^n$. Notice that
there is a monomorphism of two split extensions of algebras
\[
\xymatrix{ 1 \ar[r]\ar[d] & H_*(\Omega SU(2n+1); \Z)\ar[r]\ar[d] &
H_*(\Omega(SU(2n+1)/T^n);\Z )\ar[r]\ar[d] &H_*(T^n;\Z )\ar[r]\ar[d] &
1\ar[d]\\
1 \ar[r]& H_*(\Omega SU(2n+1); \Q)\ar[r] & H_*(\Omega(SU(2n+1)/T^n);\Q
)\ar[r] &H_*(T^n; \Q)\ar[r]&1.}
\]
Denote by $\bar{c}_2,\ldots ,\bar{c}_{2n+1}$ the universal
transgressive generators in $H^{*}(SU(2n+1);\Z )$ which map to the
symmetric polynomials $c_2=\sum\limits_{1\leq i<j\leq
2n+1}\hat{x}_i\hat{x}_j,\ldots ,c_{2n+1}=\hat{x}_1\cdots \hat{x}_{2n}\hat{x}_{2n+1}$ generating
$H^{*}(BSU(2n+1);\Z)$.  Let $T^{2n}$ be a maximal torus which contains $T^n$. The elements  $\hat{x}_1,\ldots ,\hat{x}_{2n}, \hat{x}_{2n+1}$ are
the integral generators of $H_{*}(T^{2n}; \Z)$ and
$\sum\limits_{i=1}^{2n+1}x_i=0$.  Recall from~\cite{TMZ} that $\hat{x}_{2i+1}=0$ when restricted to $T^n$ so that $\hat{x}_{2i}$ are integral generators in $H^{*}(T^n; \Z )$, where $1\leq i\leq n$.  We denote these generators $\hat{x}_{2i}$ by $x_i$, where $1\leq i\leq n$.

Now let $y_1,\ldots ,y_{2n}$ be the
integral generators of $H_{*}(\Omega SU(2n+1);\Z )$ obtained by the
transgression of the elements from $H_{*}(SU(2n+1);\Z )$ which are
the duals of $\bar{c}_2,\ldots ,\bar{c}_{2n+1}$. Further,
the set of primitive elements in $H_{*}(\Omega SU(2n+1);\Z )$ is
spanned by the elements $\sigma _1,\ldots ,\sigma _{2n}$ which can be
expressed in terms of $y_1,\ldots, y_{2n}$ using the Newton formula
\begin{equation}\label{Newton}
\sigma_{k}=\sum_{i=1}^{k-1}(-1)^{i-1}\sigma_{k-i}y_{i}+(-1)^{k-1}ky_{k},
\;\; 1\leq k\leq 2n.
\end{equation}
The integral elements $\sigma _1,\ldots ,\sigma _{2n}$ rationalise to
the elements $b_1, c_1,\ldots ,b_n,c_n\in H_{*}(\Omega SU(2n+1);\Q )$. The
generators $a_1,\ldots,a_n$ in $H_*(T^n;\Q )$ are the rationalised
images of the integral generators $x_1,\ldots, x_n$ in $H_{*}(T^n;\Z
)$. To decide the integral extension, we consider the rational
Pontrjagin ring structure~\eqref{QSU(2n+1)} of $\Omega (SU(2n+1)/T^n)$.
Looking at the  above commutative diagram of the algebra extensions,
we conclude that the integral elements
\[\begin{array}{ll}
x_kx_l+x_lx_k  & \text{ for }  1\leq k,l\leq n, k\neq l, \\
x_k^2-\sigma_1 & \text{ for }  1\leq k\leq n,\\
x_k\sigma_l-\sigma_lx_k &  \text{ for }  1\leq k\leq n, 1\leq l\leq 2n \\
\sigma_k\sigma_l-\sigma_l\sigma_k &  \text{ for }  1\leq k,l\leq 2n
\end{array}
\]
from $H_{*}(\Omega(SU(2n+1)/T^n);\Z )$ map to zero in
$H_*(\Omega(SU(2n+1)/T^n);\Q)$. As the map between the algebra
extensions is a monomorphism, we conclude that these integral
elements are zero. Using that there is no torsion in homology and
Newton formula~\eqref{Newton}, we have
\[
\begin{array}{ll}
x_kx_l+x_lx_k=0 & \text{ for } 1\leq k,l\leq n, k\neq l, \\
x_k^2=y_1 & \text{ for }  1\leq k\leq n,\\
x_ky_l-y_lx_k=0&  \text{ for }  1\leq k,l\leq n, \\
x_kz_l-z_lx_k=0&  \text{ for }  1\leq k,l\leq n, \\
y_ky_l-y_ly_k=0 &  \text{ for }  1\leq k,l\leq n\\
z_ky_l-y_lz_k=0 &  \text{ for }  1\leq k,l\leq n\\
z_kz_l-z_lz_k=0 &  \text{ for }  1\leq k,l\leq n
\end{array}
\]
which completely describes the integral Pontrjagin ring of
$\Omega(SU(2n+1)/T^n)$.
\end{proof}

In the analogues way we have the following result.
\begin{theorem}
\label{SU2n integral} The integral Pontrjagin homology ring of the loop
space on $SU(2n)/T^n$ is
\[
H_*\left(\Omega (SU(2n)/T^n); \Z\right)\cong
\]
\[
\left( T(x_1,\ldots ,x_n)\otimes \Z[y_2,\ldots
,y_n, z_1,\ldots, z_{n-1}]\right)/\left\langle x_1^2=\ldots =x_n^2 , x_px_q+x_qx_p \text{ for }
1\leq p,q\leq n, p\neq q\right\rangle
\]
where the generators $x_1,\ldots ,x_n$ are of degree $1$, the
generators $y_i$ are of degree $4i-2$ for $2\leq i\leq n$,  and the
generators $z_i$ are of degree $4i$ for $1\leq i\leq n-1$.\qed
\end{theorem}

\subsection{The integral homology of $\Omega (SO(2n+2)/T^n)$ and $\Omega (SO(8)/T^2)$}
\begin{lemma}
\label{equivalence}
For $k\leq\rk SO(m)$, 
\[
Spin(m)/ T^k\simeq SO(m)/T^k.
\]
\end{lemma}
\begin{proof}
Consider the extended diagram of fibrations and cofibrations
\[
\xymatrix{ F\ar[d]\ar[r] & S^1\ar[r]\ar[d] & S^1\ar[d]\\
\mathbb{Z}/2\ar[r]\ar[d]& Spin(m)\ar[d]\ar[r] & SO(m)\ar[d]\\
K\ar[r] & Spin(m)/ S^1\ar[r] & SO(m)/S^1}
\]
where $F$ and $K$ are the homotopy fibers of the  corresponding horizontal maps.
Since $SO(m)/T^k$ is simply connected for any $k\leq\rk SO(m)$, we conclude that $F\simeq \mathbb{Z}/2$ and thus $K\simeq *$ proving that $ Spin(m)/ S^1\simeq SO(m)/S^1$ and from there that $ Spin(m)/ T^k\simeq SO(m)/T^k$ for $k\leq\rk SO(m)$.
\end{proof}
\begin{theorem}
The integral Pontrjagin homology ring of the loop
space on $SO(2n+2)/T^n$ is
\[
H_*\left(\Omega (SO(2n+2)/T^n); \Z\right)\cong
\]
\[
\left( T(x_1,\ldots ,x_n)\otimes \Z [y_1,\ldots, y_{n-1},y_{n}+z,
y_{n}-z,2y_{n+1},\ldots ,2y_{2n}]\right)/I
\]
where $I$ is generated by
\[
\begin{array}{l}x_1^2-y_1,x_i^2-x_{i+1}^2\; \text{ for }\; 1\leq i
\leq n-1\\
x_kx_l+x_lx_k\; \text{ for }\; k\neq l\\
y_i^2-2y_{i-1}y_{i+1}+2y_{i-2}y_{i+2}-\ldots \pm 2y_{2i}\;
 \text{ for }\; 1\leq i\leq n-1\\
(y_{n}+z)(y_{n}-z)-2y_{n}y_{n+2}+\ldots \pm 2y_{2n}
\end{array}
\]
where $\deg x_i =1$ for $1\leq i\leq n$, $\deg y_i=2i$ for $1\leq
i\leq n-1$, $\deg(y_{n}+z)=\deg(y_{n}-z)=2n$, $\deg
2y_{i}=2i$ for $n+1\leq i\leq 2n$ and $y_0=1$.
\end{theorem}
\begin{proof}
Since $SO(2n+2)/T^n\simeq Spin(2n+2)/T^n$ by Lemma~\ref{equivalence}, we have that  $\Omega
(SO(2n+2)/T^n)\cong \Omega (Spin (2n+2)/T^n)$. It is known that
$\Omega Spin(2n+2)\cong\Omega _{0}SO(2n+2)$, see for
example~\cite{MT}.

Recall from~\cite{BT} that the algebra $H_{*}(\Omega _{0}SO(2n+2);\Z)$
is generated by the elements
$y_1,\ldots,y_{n-1}, y_{n}+z,y_{n}-z, 2y_{n+1},\ldots ,2y_{2n}$ which satisfy the relations
\[
y_i^2-2y_{i-1}y_{i+1}+2y_{i-2}y_{i+2}-\ldots \pm 2y_{2i} =0
\text{ for } 1\leq i\leq n-1,
\]
\[
(y_{n}+z)(y_{n}-z)-2y_{n}y_{n+2}+\ldots \pm 2y_{2n}=0.
\]
These relations eliminate $2y_{2i}$ as
generators for $[\frac{n+2}{2}]\leq i\leq n-1$, while for $1\leq
i\leq [\frac{n+2}{2}]-1$, they induce new relations on $y_{2i}$
implying that $y_{2i}$ are generators only in the homology of
$\Omega_{0}SO(2n+2)$ with coefficients where $2$ is not invertible.
The subspace of primitive elements in $H_{*}(\Omega _{0}SO(2n+2);\Z )$
is spanned by the elements $p_1,p_3,\ldots p_{n-1},2z,2p_{n+1},\ldots
,2p_{2n-1}$ for $n$ odd and by the elements
$p_1,p_3,\ldots,p_{n},2z,2p_{n+2},\ldots ,2p_{2(n-1)+1}$ for $n+1$
even. These primitive generators are obtained by transgressing the
elements in $H_{*}(SO(2n+2);\Z )$ which are the Poincare duals of the
universal transgressive generators $\bar{\sigma}_1,\ldots
,\bar{\sigma}_{n},\bar{\lambda}$  in $H^{*}(SO(2n+2);\Z )$. The generators
$\bar{\sigma}_1,\ldots ,\bar{\sigma}_{n},\bar{\lambda}$ map to the polynomials
$\sigma _{i}(x_1^2,\ldots ,x_{n+1}^2)$ for $1\leq i\leq n$ and
$\lambda =x_1\cdots x_{n+1}$ which generate the free part in
$H^{*}(BSO(2n+2);\Z)$. here $x_1,\ldots ,x_n$ are the generators for $H_{*}(T^{n+1}; \Z )$ where $T^{n+1}$ is the maximal torus for $SO(2n+2)$.
Recall from~\cite{TFP} that in the case of generalised symmetric space $SO(2n+2)/T^n$ we have
that $T^n$ is embedded in $T^{n+1}$ in a such a way that $x_{n+1}=0$ on $T^n$ and thus gives the morphism $H_{*}(\Omega _{0}SO(2n+2);\Q )\to H_{*}(\Omega (SO(2n+2)/T^n);\Q )$.  

The proof of the theorem is now analogue to the proof of Theorems 4.3 an 4.4 in~\cite{GT}. We consider the morphism of two extensions of algebras
\[
\xymatrix{H_*(\Omega _{0}SO(2n+2); \Z)\ar[r]\ar[d] &
H_*(\Omega(SO(2n+2)/T^n);\Z )\ar[r]\ar[d] &H_*(T^n;\Z )\ar[d]\\
H_*(\Omega_0SO(2n+2); \Q)\ar[r] & H_*(\Omega(SO(2n+2)/T^n);\Q
)\ar[r] &H_*(T^n; \Q).}
\]
and taking into account the rational homology calculations of $\Omega(SO(2n+2)/T^n)$ in Theorem~\ref{SO(2n+2)} and the above description of the algebra $H_{*}(\Omega _{0}SO(2n+2);\Z)$ we come to the result. 
\end{proof}

\begin{theorem}
 The integral Pontrjagin homology ring of the loop
space on $SO(8)/T^2$ is
\[
H_*\left(\Omega (SO(8)/T^2); \Z\right)\cong
\]
\[
\left( T(x_1,x_2)\otimes \Z [y_1, y_{2},y_{3}+z,
y_{3}-z,2y_{4},2y_{5},2y_{6}]\right)/I
\]
where $I$ is generated by
\[
\begin{array}{l}x_1^2-y_1,x_1^2-x_{2}^2\\
x_1^2-x_1x_2+x_2x_1\\
y_i^2-2y_{i-1}y_{i+1}+2y_{i-2}y_{i+2}-\ldots \pm 2y_{2i}\;
 \text{ for }\; 1\leq i\leq 2\\
(y_3+z)(y_3-z)-2y_3y_5+2y_6
\end{array}
\]
where $\deg x_1=\deg x_2 =1$, $\deg y_1=2$, $\deg y_2=4$, $\deg(y_3+z)=\deg(y_3-z)=6$, $\deg
2y_{i}=2i$ for $4\leq i\leq 6$ and $y_0=1$.\qed
\end{theorem}
\begin{remark}
It follows from~\cite{TFP} that for this generalised symmetric space $T^2$ is embedded  in the maximal torus $T^4$ for $SO(8)$ in such a way that  the canonical integral generators $x_1,x_2,x_3,x_4$ for $H_{*}(T^4;\Z )$ restrict on $T^2$ to $x_4=0$ and $x_1=x_2+x_3$. We took this  into account when we considered  the  morphism  $H_{*}(\Omega _{0}SO(8);\Q )\to H_{*}(\Omega (SO(8)/T^2);\Q )$. 
\end{remark}
\subsection{The integral homology of $\Omega (E_6/T^4)$}
The integral homology algebra $H_{*}(\Omega E_6;\Z )$ is described
in~\cite{N} and it is given by
\[
H_{*}(\Omega E_6;\Z )\cong \Z
[y_1,y_2,y_3,y_4,y_5,y_7,y_8,y_{11}]/\langle y_1^2-2y_2,
y_1y_2-3y_3\rangle \;
\]
where $\deg y_{i}=2i$ for $i=1,2,3,4,5,7,8,11$.

Using the same argument as for the previous cases, we deduce the
integral Pontrjagin homology of the based loop space on $E_6/T^4$.
\begin{theorem}\label{intE6/T}
The integral Pontrjagin homology ring of $\Omega (E_6/T^4)$ is given
by
\[
H_{*}(\Omega (E_6/T^4); \Z )\cong
\]
\[
\left(T(x_1,x_2,x_3,x_4)\otimes \Z
[y_1,y_2,y_3,y_4,y_5,y_7,y_8,y_{11}]\right)/I
\]
where $I= \langle x_k^2=12y_1 \text{ for } 1\leq
k\leq 4, \; x_px_q+x_qx_p \text{ for } 1\leq p,q\leq 4,\; 2y_2=x_1^4,\; 3y_3=x_1^2y_2 \rangle$ and where $\deg
x_i=1$ for $1\leq i\leq 4$, and $\deg y_i=2i$ for
$i=1,2,3,4,5,7,8,11$.
\end{theorem}

\section{Appendix}
In the Appendix we illustrate that our methods could be applied not only to generalised symmetric spaces having toral stationary subgroup but wider. We concentrate on the example of $U(n)/(T^k\times U(n-k))$ for $k\leq n-2$.

\begin{lemma}\label{more}
The homology groups $H_*(\Omega (U(n)/T^l\times U(k));\mathbb Z)$, $l+k\leq n$, are torsion free.
\end{lemma}
\begin{proof}
As $U(n)/U(k)$ is simply connected and $H_*(\Omega (U(n)/U(k);\Z )\cong \Z [x_{2k},\ldots ,x_{2(n-1)}]$, the statement follows by Proposition~\ref{simplyconnected}.
\end{proof}  

\begin{theorem}\label{more-rational}
The rational homology ring of the loop space on the 
manifold $U(n)/(T^k\times U(n-k))$, $k\leq n-2$ is
\begin{equation}\label{partial} 
H_*(\Omega ( U(n)/(T^k\times U(n-k))); \Q) \cong \wedge (a_1,\ldots ,a_k)
\otimes \Q[b_{n-k+1},\ldots ,b_n],
\end{equation}
where the generators $a_{i}$ are of degree $1$ for $1\leq i\leq k$, and
the generators $b_{j}$ are of degree $2j-2$  for $n-k+1\leq j\leq n$.
\end{theorem}
\begin{proof}
It follows from~\cite{TFP} and~\cite{TMZ} that the rational cohomology for this space is given by
\[
H^{*}(U(n)/(T^k\times U(n-k));\Q )\cong \Q [x_1,\ldots ,x_k]/ \left\langle P_{n-k+1},\ldots P_n \right\rangle 
\]
where $P_j=\sum_{i=1}^{j}x_i^j$ for $n-k+1\leq j\leq n$. Since $\deg (P_j)=2j\geq 2(n-k+1)\geq 6$ we conclude that the differential of the minimal model for this spaces has no quadratic part. Therefore the given procedure for the computation of the corresponding loop space homology leads to formula~\eqref{partial}.
\end{proof}

\begin{theorem}
The integral Pontrjagin homology ring of the loop space on  $U(n)/(T^k\times U(n-k))$, $k\leq n-2$ is
\begin{equation}\label{partialinteg} 
H_*(\Omega ( U(n)/(T^k\times U(n-k))); \Z) \cong \wedge (x_1,\ldots ,x_k)
\otimes \Z[y_{n-k+1},\ldots ,y_n]
\end{equation}
where the generators $x_{i}$ are of degree $1$ for $1\leq i\leq k$ and
the generators $y_{j}$ are of degree $2j-2$  for $n-k+1\leq j\leq n$.
\end{theorem}
\begin{proof}
It follows from Lemma~\ref{more} and  rational computation given in Theorem~\ref{more-rational} that the  fibration
$\Omega (U(n)/U(n-k))\to \Omega (U(n)/(T^k\times U(n-k)))\to T^k$  induces the monomorphism of two split algebra extensions
\[
\xymatrix{ 1 \ar[r]\ar[d] & H_*(\Omega (U(n)/U(k)); \Z)\ar[r]\ar[d] &
H_*(\Omega(U(n)/(T^k\times U(n-k)));\Z )\ar[r]\ar[d] &H_*(T^k;\Z )\ar[r]\ar[d] &
1\ar[d]\\
1 \ar[r]& H_*(\Omega(U(n)/U(k)); \Q)\ar[r] & H_*(\Omega(U(n)/(T^k\times U(n-k)));\Q
)\ar[r] &H_*(T^k; \Q)\ar[r]&1.}
\]

As all the Whitehead products which define the extension are trivial, the theorem follows at once.
\end{proof}

\bibliographystyle{amsalpha}

\end{document}